\documentclass[reqno,12pt]{amsart}
\usepackage{amsfonts}
\usepackage{amssymb}
\usepackage{amsthm}
\usepackage{amsmath}
\usepackage{mathtools}
\usepackage{graphicx}
\usepackage{caption}
\usepackage{subcaption}
\usepackage[percent]{overpic}

\author{Michael Pfeuti}
\title{Twist Triviality of Canonical Seifert Surfaces}

\date{}

\theoremstyle{plain}
\newtheorem{theorem}{Theorem}
\newtheorem{corollary}{Corollary}
\newtheorem{proposition}{Proposition}
\newtheorem{lemma}{Lemma}

\theoremstyle{definition}
\newtheorem{definition}{Definition}

\theoremstyle{remark}

\newtheorem*{remark}{Remark}

\let\ORIincludegraphics\includegraphics
\renewcommand{\includegraphics}[2][]{\ORIincludegraphics[scale=0.6,#1]{#2}}

\graphicspath{{figures/}}

\begin{document}

\begin{abstract}
We generalize the idea of unknotting knots to Seifert surfaces. We define an operation called ribbon twist which serves as the equivalent of a crossing change for knots. A Seifert surface is considered untwisted, the equivalent to unknotted, if it is isotopic to a standardly embedded n-fold punctured torus. A Seifert surface is said to be twist trivial if it can be untwisted by ribbon twists. We show that canonical Seifert surfaces are twist trivial.
\end{abstract}
\maketitle

\section{Introduction}
A prevalent method for studying a link $L$ is to analyze its \emph{Seifert surfaces}, i.e. compact, connected, orientable surfaces $\Sigma$ that have the link as their boundaries $\partial \Sigma = L$. The well known fact that every link has a Seifert surface was shown by Seifert in \cite{Seifert1935} by presenting the \emph{Seifert algorithm}. A \emph{canonical Seifert surface} is a Seifert surface that can be obtained by the Seifert algorithm.

\begin{figure}[b]
  \centering
  \begin{subfigure}[t]{0.22\textwidth}
    \centering
    \includegraphics[scale=1.5]{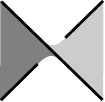}
    \caption{}
    \label{fig:crossing_change_surface_a}
  \end{subfigure}
  \begin{subfigure}[t]{0.22\textwidth}
    \centering
    \includegraphics[scale=1.5]{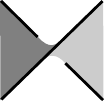}
    \caption{}
    \label{fig:crossing_change_surface_b}
  \end{subfigure}
  \begin{subfigure}[t]{0.22\textwidth}
    \centering
    \includegraphics[scale=1.5]{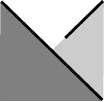}
    \caption{}
    \label{fig:crossing_change_surface_c}
  \end{subfigure}
  \begin{subfigure}[t]{0.22\textwidth}
    \centering
    \includegraphics[scale=1.5]{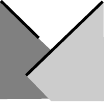}
    \caption{}
    \label{fig:crossing_change_surface_d}
  \end{subfigure}
  \caption{}
\end{figure}

It is true that every link $L$ can be unknotted by $u(L)$ crossing changes (Chapter 3.1 in \cite{Adams1994}) but how does a crossing change of $L$ affect a Seifert surface $\Sigma$ of this link? To understand how the surface changes, it suffices to look at the effect of a crossing change locally. A crossing change involves two strands. Now, there are two cases. In the first case, the surface forms a ribbon between the two strands. This ribbon features a half twist caused by the crossing of the two strands (Figure \ref{fig:crossing_change_surface_a}). The effect of a crossing change on the surface is a full twist such that the half twist twists in the opposite direction (Figure \ref{fig:crossing_change_surface_b}). Notice that the resulting surface is still a Seifert surface. In the second case, the surface in the vicinity of the crossing is disconnected. Therefore, there are two surface patches, one for each strand (Figure \ref{fig:crossing_change_surface_c}). Here, a crossing change causes the surface patches to be swapped. In general, this causes the surface to be self-intersecting and the resulting surface ceases to be a Seifert surface (Figure \ref{fig:crossing_change_surface_d}). 

As the latter of the two cases does not lead to a Seifert surface, we want to exclude such crossing changes. Therefore, we generalize the crossing changes of the first case by allowing an operation of inserting full twists into any ribbon. 

\begin{definition}
A \emph{ribbon twist} is a cut and glue operation on a Seifert surface $\Sigma$. Let $I = \varphi([0,1]) $ be an embedded interval such that $\varphi(0)$, $\varphi(1) \in \partial \Sigma$ and $\varphi((0,1)) \in \Sigma \setminus \partial \Sigma$  (Figure \ref{fig:ribbon_twist_a}). Cut along $I$, insert a full twist on one side (Figure \ref{fig:ribbon_twist_b}) and glue both sides back together along $I$ (Figure \ref{fig:ribbon_twist_c}). 
\end{definition}

\begin{figure}[htp]
  \centering
  \begin{subfigure}[t]{0.3\textwidth}
    \centering
    \begin{overpic}{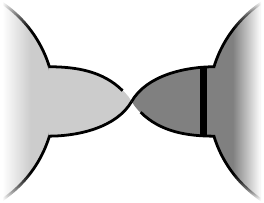}
      \put(70, 55){$I$}
    \end{overpic}
    \caption{}
    \label{fig:ribbon_twist_a}
  \end{subfigure}
  \begin{subfigure}[t]{0.3\textwidth}
    \centering
    \includegraphics{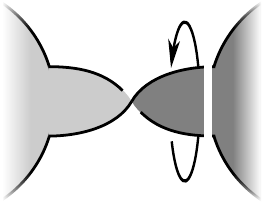}
    \caption{}
        \label{fig:ribbon_twist_b}
  \end{subfigure}
  \begin{subfigure}[t]{0.3\textwidth}
    \centering
    \begin{overpic}{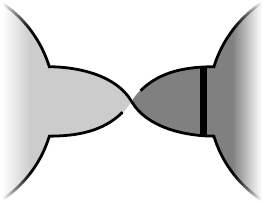}
      \put(70, 55){$I$}
    \end{overpic}
    \caption{}
    \label{fig:ribbon_twist_c}
  \end{subfigure}
  \caption{}
  \label{fig:ribbon_twist}
\end{figure}

Based on ribbon twists, we can define an equivalence relation for Seifert surfaces.

\begin{definition}
Let $\Sigma_1$ and $\Sigma_2$ be two Seifert surfaces. The surfaces $\Sigma_1$ and $\Sigma_2$ are called \emph{twist equivalent} if $\exists R_1, \dots, R_n$ ribbon twists such that $R_n \circ ... \circ R_1 (\Sigma_1)$ is isotopic to $\Sigma_2$.
\end{definition}

Now, we pose the question if a link along with its Seifert surface can be unknotted by isotopies and ribbon twists. Foremost, we would like to know if the link can be unknotted with isotopies and ribbon twists such that the resulting surface is a standard n-fold punctured torus. Notice, it suffice to consider if the surface can be transformed into a standard n-fold punctured torus since the associated link is automatically unknotted in the process. We focus on the standard n-fold punctured torus because this surface is in a sense unknotted. 
\begin{figure}[b]
  \centering
  \includegraphics{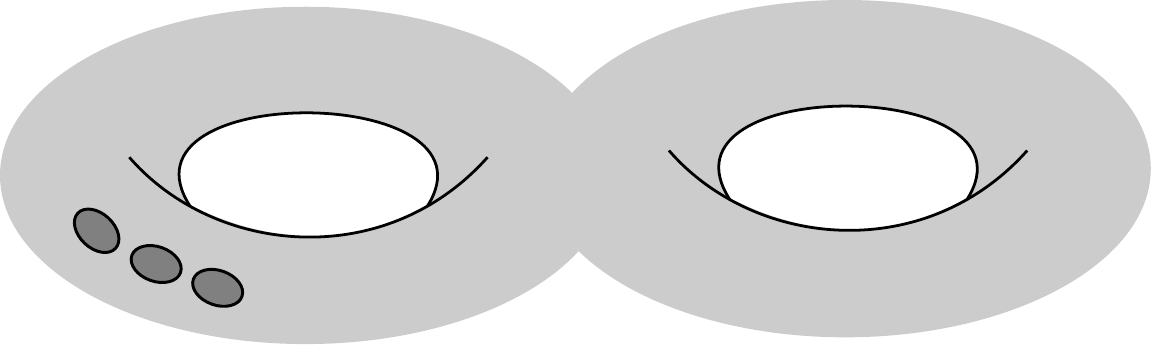}
  \caption{}
  \label{fig:twist_trivial}
\end{figure}
\begin{definition}
A Seifert surface is called \emph{twist trivial} if it is twist equivalent to a standardly embedded n-fold punctured torus, i.e. a subsurface of the boundary of an unknotted handlebody in $\mathbb{R}^3$ (Figure \ref{fig:twist_trivial}).
\end{definition}

The first result is a necessary condition for a Seifert surface to be twist trivial.
\begin{proposition}
\label{prop:twist_trivial_free}
If a Seifert surface $\Sigma$ is twist trivial then $\pi_1(\mathbb{R}^3 \setminus \Sigma)$ is a free group.
\end{proposition}
The main result of this article is the fact that any canonical Seifert surface of any link is twist trivial. 
\begin{theorem}
\label{thm:canonical_surface_twist_trivial}
If $\Sigma$ is a canonical Seifert surface then $\Sigma$ is twist trivial.
\end{theorem}
This theorem provides a large class of Seifert surfaces which are twist trivial. However, we do not yet know if this class can be enlarged. For instance, are all fiber surfaces (\cite{Stallings1978,Rolfsen1976}) twist trivial? Baader and Dehornoy answered this question partially. They showed that fiber surfaces of positive braid knots are twist trivial (\cite{Baader2013}). Yet, the most general question to ask is if the necessary condition shown in Proposition \ref{prop:twist_trivial_free} is also a sufficient condition. In other words, is every Seifert surface with a free fundamental group of its complement twist trivial? We do not yet know the answer.

Furthermore, we define in analogy to the unknotting number of links the untwisting number of twist trivial surfaces.
\begin{definition}
Let $\Sigma$ be a twist trivial surface. The \emph{untwisting number} of $\Sigma$ is $ut(\Sigma)$ the minimal number of ribbon twists required to untwist the surface, that is
\begin{equation*}
ut(\Sigma) = \min \left\{n\in \mathbb{N} | R_n \circ \dots \circ R_1(\Sigma) \text{ isotopic to } \Sigma_1\right\}
\end{equation*}
where $R_1, \dots, R_n$ are ribbon twists and $\Sigma_1$ the standard n-fold punctured torus.
\end{definition}
As a corollary of Theorem \ref{thm:canonical_surface_twist_trivial} we obtain an upper bound for the untwisting number of canonical Seifert surfaces.
\begin{corollary}
\label{cor:upper_bound}
Let $\Sigma=\left\{D_1,\dots,D_n;B_1,\dots,B_m;T_1,\dots,T_o\right\}$ be a canonical Seifert surface. Then,
\begin{equation*}
ut(\Sigma) \leq 6 \cdot \sum\limits_{i=1}^{m-1} i.
\end{equation*}
\end{corollary}

\textbf{Acknowledgements.} I would like to thank Prof. Dr. Sebastian Baader for introducing me to this topic and for all the insightful discussions and ideas. Furthermore, I would like to thank Filip Misev, Livio Liechti and Luca Studer for their helpful advice and discussions.

\section{Seifert Surfaces}
\label{sec:seifert_surface}
The fundamental fact that every link has a Seifert surface was proven by Seifert in 1935 (\cite{Seifert1935}). He presented an algorithm, now known as the Seifert algorithm, which allows the construction of a Seifert surface for any link. It laid the foundation for numerous methods for studying knots. For instance, a Seifert surface is required to define the associated Seifert matrix, which is then required for the definition of important knot invariants such as the knot signature or the Alexander polynomial (\cite{Alexander1928,Murasugi1965, Cromwell2004}).

The Seifert algorithm produces a set of disjoint closed curves. These curves are known as \emph{Seifert circles}. We distinguish two types of Seifert circles like \cite{Kauffman1987, Aaltonen2014}. A \emph{type I Seifert circle} does not enclose nor is it enclosed in any other Seifert circle. A \emph{type II Seifert circle} does enclose or is enclosed in at least one other Seifert circle. The Seifert circles are then the basis for a set of discs. These discs are bounded by Seifert circles and called \emph{Seifert discs}. Likewise, if the Seifert disc is bounded by a type I Seifert circle (respectively type II) we call it \emph{type I Seifert disc} (respectively \emph{type II}). The Seifert discs are connected by half twisted bands to produce the Seifert surface. Such a half twisted band is called \emph{positive} (respectively \emph{negative}) if it was obtained from a positive (respectively negative) crossing. A canonical Seifert surface is constructed from Seifert discs, half twisted bands and unknotted tubes. So, we usually interpret a canonical Seifert surface $\Sigma$ as a set of Seifert disc $\left\{D_1,\dots,D_n\right\}$, a set of half twisted bands $\left\{B_1,\dots,B_m\right\}$ and a set of unknotted tubes $\left\{T_1,\dots,T_o\right\}$. We write $\Sigma = \left\{D_1,\dots,D_n;B_1,\dots,B_m; T_1,\dots,T_o\right\}$ for a canonical Seifert surface $\Sigma$. The number of half twisted bands attached to a Seifert disc $D_i$ is called the \emph{degree of $D_i$} and is denoted as $deg(D_i)$. We write $B(D_i)$ for the set of half twisted bands attached to disc $D_i$. If the unknotted tubes $\left\{T_1,\dots,T_o\right\}$ are removed from a canonical Seifert surface we are left with $o+1$ connected components. We refer to such a connected component simply as a \emph{component} of the canonical Seifert surface. In the following we derive certain results which are only true for individual components of a canonical Seifert surface. In this case we write $\Sigma = \left\{D_1,\dots,D_n;B_1,\dots,B_m; -\right\}$, which indicates that there are no unknotted tubes and therefore the surface consists of only one component. Lastly, two half twisted bands $B_k, B_l$ are called \emph{adjacent} if $B_k$ and $B_l$ connect the same Seifert discs, $B_k, B_l \in B(D_i) \cap B(D_j)$, and there does not exist any band between $B_k$ and $B_l$ (Figure \ref{fig:adjacent_band_a}). Figure \ref{fig:adjacent_band_b} shows an example of non-adjacent bands.
\begin{figure}[t]
  \centering
  \begin{subfigure}[t]{0.3\textwidth}
    \centering
    \begin{overpic}{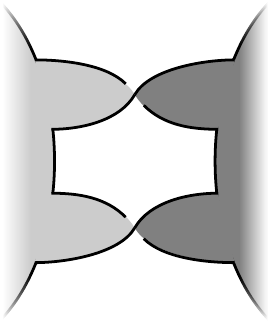}
      \put(-8,45){$D_i$}
      \put(75,45){$D_j$}
      \put(33,5){$B_k$}
      \put(33,83){$B_l$}
    \end{overpic}
    \caption{}
    \label{fig:adjacent_band_a}
  \end{subfigure}
  \begin{subfigure}[t]{0.3\textwidth}
    \centering
    \begin{overpic}{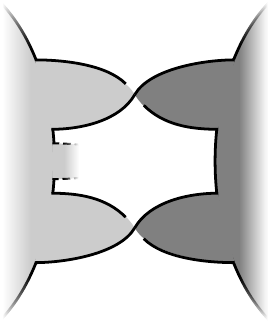}
      \put(-8,45){$D_i$}
      \put(75,45){$D_j$}
      \put(33,5){$B_k$}
      \put(33,83){$B_l$}
    \end{overpic}
    \caption{}
    \label{fig:adjacent_band_b}
  \end{subfigure}
  \caption{}
\end{figure}

An important observation we can make from the Seifert algorithm is that if there are no type II Seifert circles the resulting canonical Seifert surface is \emph{almost planar}. This means that the surface can be embedded in a plane except for the crossing regions in the half twisted bands and the unknotted tubes. In the following we will make use of this observation several times. Therefore, it is beneficial to us to have type II free canonical Seifert surfaces. Kauffman showed in Proposition 7.3 in \cite{Kauffman1987} that a link diagram with type II Seifert circles can be rid of all type II Seifert circles by adding one additional link component per type II Seifert circle. Hirasawa improved this result by showing that the type II Seifert circles of a link $L$ can be removed by an isotopy instead of adding link components. Moreover, the canonical Seifert surface of $L$, obtained from a given projection, is also deformed by the isotopy applied to $L$. The resulting Seifert surface remains canonical for the isotoped link with respect to the same projection. Aaltonen republished Hirasawa's result in 2014.

\begin{theorem}[\cite{Aaltonen2014,Hirasawa1995}]
\label{thm:altonen}
Let $\Sigma$ be a canonical Seifert surface with type II Seifert discs. Then, $\Sigma$ is isotopic to a type II free canonical Seifert surface $\Sigma'$.
\end{theorem}
\begin{proof}
We present the proof from \cite{Aaltonen2014} in a less technical form than the proof in the publication but the idea is the same. The proof is an induction over the number of type II Seifert discs. It suffices to show that any type II Seifert disc can be removed by an isotopy. The description of the following isotopies is with respect to the projection from which $\Sigma$ was obtained. Yet, the isotopies are executed in $\mathbb{R}^3$.

\begin{figure}[t]
  \centering
  \begin{subfigure}[t]{0.45\textwidth}
    \centering
    \begin{overpic}[scale=1.25]{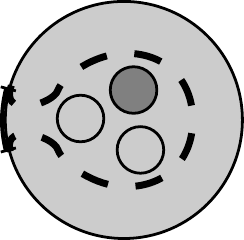}
      \put(40, 79){$D_i$}
      \put(-17, 40){$I$}
    \end{overpic}
    \caption{}
    \label{fig:thm_aaltonen_a}
  \end{subfigure}
  \begin{subfigure}[t]{0.45\textwidth}
    \centering
    \begin{overpic}[scale=1.25]{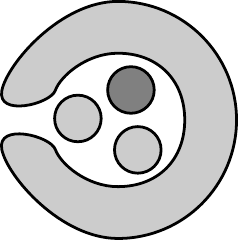}
      \put(40, 80){$D_i$}
    \end{overpic}
    \caption{}
    \label{fig:thm_aaltonen_b}
  \end{subfigure}
  \begin{subfigure}[t]{0.45\textwidth}
    \centering
    \begin{overpic}{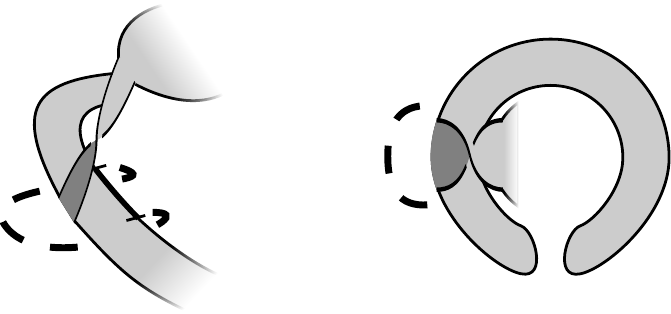}
      \put(20, 17){$J$}
    \end{overpic}
    \caption{}
    \label{fig:thm_aaltonen_c}
  \end{subfigure}
  \begin{subfigure}[t]{0.45\textwidth}
    \centering
    \begin{overpic}{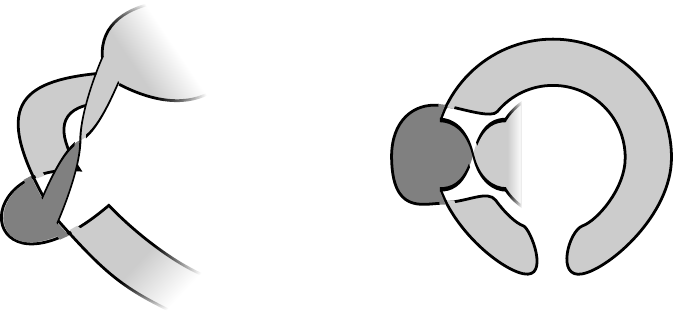}
    \end{overpic}
    \caption{}
    \label{fig:thm_aaltonen_d}
  \end{subfigure}
  \caption{}
\end{figure}

Let $D_i$ be a type II Seifert disc. The bands in $B(D_i)$ which connect to discs inside $D_i$ with respect to the projection are called inward bands. Outward bands are the ones which connect to discs outside $D_i$. Choose an interval $I$ on $\partial D_i$ such that it does not intersect with any band in $B(D_i)$. Drag the interval $I$ into the interior of $D_i$ such that $I$ traces along $\partial D_i \setminus I$ so closely that the dragged interval does not intersect with any interior disc in the projection (Figure \ref{fig:thm_aaltonen_a} and \ref{fig:thm_aaltonen_b}). Now, for every inward band $B_j$ drag an interval $J \subset I$ below $B_j$ to the exterior of $D_i$. The dragging is performed around $B_j \cap D_i$ in a rotational motion and in the direction that $J$ never intersects with $B_j$ (Figure \ref{fig:thm_aaltonen_c} and \ref{fig:thm_aaltonen_d}). This gives rise to two new bands and two new Seifert discs. After applying such an isotopy for all inward bands in $B(D_i)$ the disc $D_i$ itself is replaced with type I Seifert discs and half twisted bands. Notice that no special treatment is necessary for outward bands and that the new discs are the Seifert discs of the isotoped link. 

In conclusion, stretching the interval $I$ into $D_i$ except for inward bands where the interval is outside of $D_i$ removes the type II Seifert disc $D_i$ at the cost of adding two new bands and two new type I Seifert discs for every inward band. Furthermore, the resulting Seifert surface is still canonical with respect to the same projection.

Figure \ref{fig:aaltonen_ex} illustrates the algorithm with an example type II Seifert disc containing only one other Seifert disc.
\end{proof}

\begin{figure}[h]
  \centering
  \begin{subfigure}[t]{0.3\textwidth}
    \centering
    \includegraphics[scale=.9]{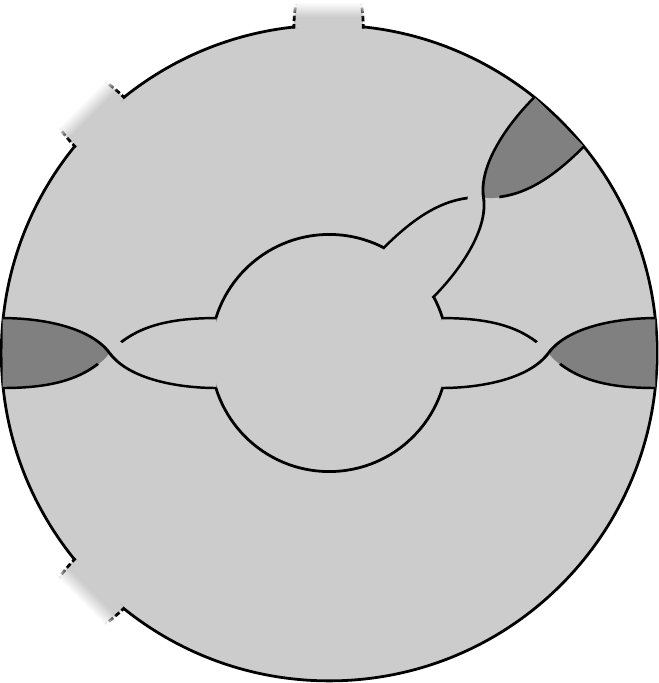}
  \end{subfigure}
  \begin{subfigure}[t]{0.3\textwidth}
    \centering
    \includegraphics[scale=.9]{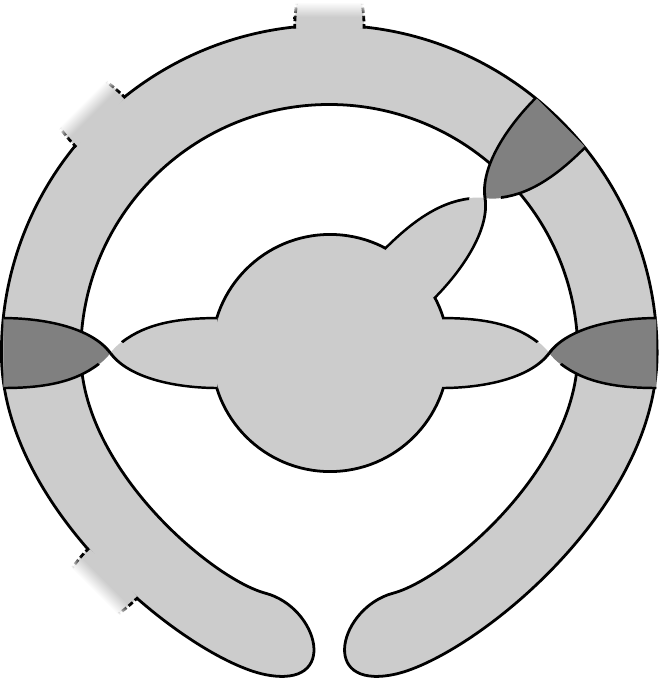}
  \end{subfigure}
  \begin{subfigure}[t]{0.38\textwidth}
    \centering
    \includegraphics[scale=.9]{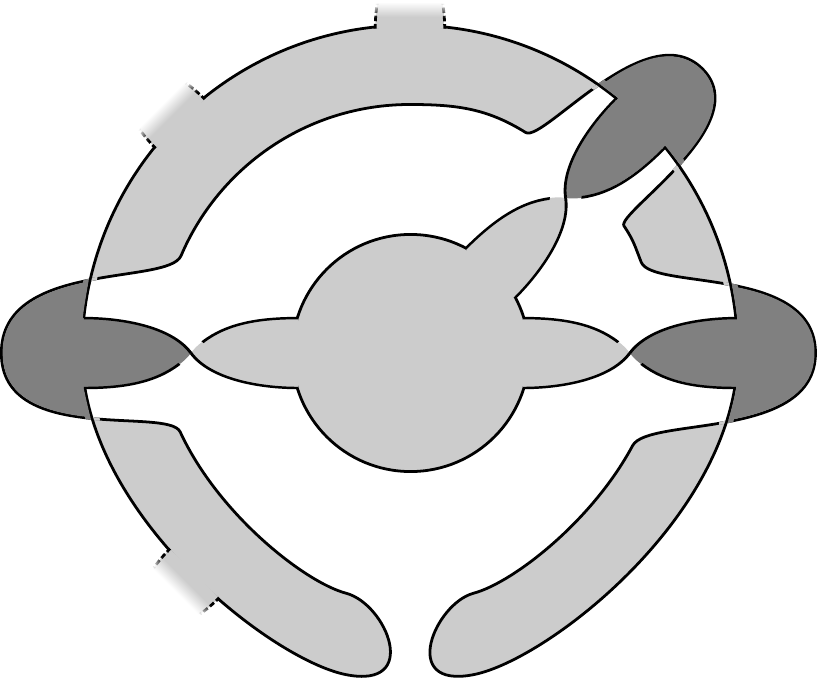}
  \end{subfigure}
  \caption{}
  \label{fig:aaltonen_ex}
\end{figure}

\section{Seifert Graphs}
\label{sec:seifert_graph}
We already mentioned that the Seifert surface is the basis for many important concepts in knot theory. A further concept based on the Seifert surface is the representation of a canonical Seifert surface as a graph. We saw that the Seifert algorithm produces a set of disjoint discs that are eventually connected by half twisted bands and possibly unknotted tubes. For every component of a canonical Seifert surface we can construct a planar graph from the discs and their connecting half twisted bands. This graph representation is helpful because we can apply well known results from graph theory to gain information about the surface.

\begin{definition}
Let $\Sigma = \left\{D_1,\dots,D_n;B_1,\dots,B_m;-\right\}$ be a canonical Seifert surface. Then, the \emph{Seifert graph} of $\Sigma$ is the graph $(V, E)$ where $V = \left\{D_1,\dots,D_n\right\}$ is the set of vertices and $E=\left\{B_1,\dots,B_m\right\}$ the set of edges.
\end{definition}

The Seifert graph of a canonical Seifert surface has several useful properties which allow us to proof Lemma \ref{lem:graph_limitation}. Lemma \ref{lem:graph_limitation} plays an essential role in the proof of Theorem \ref{thm:canonical_surface_twist_trivial}.

\begin{lemma}[\cite{Cromwell2004, Murasugi1996}]
\label{lem:seifert_graph_properties}
A Seifert graph is planar, bipartite and has no loops.
\end{lemma}

\begin{lemma}
\label{lem:graph_limitation}
Let $n > 1$ and $\Sigma = \left\{D_1,\dots,D_n;B_1,\dots,B_m;-\right\}$ be a canonical Seifert surface without type II Seifert discs. Then, either
\begin{equation*}
\exists k,l \in \left\{1,\dots,m\right\} \text{ such that } B_k, B_l \text{ are adjacent}
\end{equation*}
or 
\begin{equation*}
\exists i \in \left\{1,\dots,n\right\} \text{ such that } deg(D_i) = 1,2 \text{ or } 3.
\end{equation*}
\end{lemma}
\begin{proof}
Assume to the contrary that the type II free canonical Seifert surface has no adjacent bands and every Seifert disc has degree at least four. 

From Lemma \ref{lem:seifert_graph_properties} we know that the Seifert graph is planar, bipartite and has no loops. As the graph is planar it can be embedded in the sphere $S^2$. Such an embedding divides $S^2$ into a finite number of faces. Let $V$ be the number of vertices, $E$ the number of edges and $F$ the number of faces. It is a well known fact for the Euler characteristic (Theorem 5.A.1 in \cite{Rolfsen1976}) that 
\begin{equation}
\label{eq:euler_characteristic}
\chi(S^2) = V - E + F = 2.
\end{equation}

In the following we require that the faces stem from the Seifert graph diagram which is directly induced by the canonical Seifert surface. Furthermore, it is essential to have a canonical Seifert surface without type II Seifert discs. The reason is that without type II Seifert discs, the surface can be reconstructed from the vertices, edges and faces. Each vertex gives rise to a Seifert disc and each edge to a half twisted band (the sign is not important here). Note that this is not possible for type II discs because we lose the information about how the type II discs are connected by the half twisted bands. So, by showing that the directly induced Seifert graph diagram leads to a contradiction we can conclude that such a canonical Seifert surface could not have existed. 

Continuing with the proof, from the fact that there are no loops in the Seifert graph we conclude that every face has more than one edge. The assumption that there are no adjacent bands is equivalent with the property that every face is bounded by more than two edges. Furthermore, every face is bounded by an even number of edges because otherwise there would be a circuit of uneven length, which would violate the bipartiteness of the graph. In conclusion, every face is bounded by at least four edges. From these observations it follows that
\begin{equation}
\label{eq:face_edge_relation}
4F \leq \sum_{i=1}^{F} d(f_i) = 2E
\end{equation}
where $\left\{f_1, \dots, f_F\right\}$ is an enumeration of all faces and $d(f_i)$ is the number of edges bounding face $f_i$. The reason why the sum equals twice the edge count is the following. Each edge belongs to the boundary of two faces. Therefore, by summing over all faces each edge is counted twice.

From (\ref{eq:euler_characteristic}) and (\ref{eq:face_edge_relation}) follows
\begin{equation}
\label{eq:vertex_edge_limit}
E \leq 2V - 4.
\end{equation}

Because every Seifert disc is of degree at least four, every vertex has degree at least four. Therefore,
\begin{equation}
\label{eq:vertex_edge_relation}
4V \leq \sum_{i=1}^{V} d(v_i) = 2 E
\end{equation}
where $\left\{v_1, \dots, v_V\right\}$ is an enumeration of all vertices and $d(v_i)$ is the degree of vertex $v_i$. The reason why the sum equals twice the edge count is that each edge is incident to two vertices. Therefore, by summing over all vertices each edge is counted twice.

The inequalities (\ref{eq:vertex_edge_limit}) and (\ref{eq:vertex_edge_relation}) yield the contradiction 
\begin{equation*}
E \leq 2V - 4 \leq E -4. \qedhere
\end{equation*}
\end{proof}

\section{Twist Equivalence and Twist Triviality}
\label{sec:twist_equivalence}

It is a fact that not every Seifert surface is twist trivial. A necessary condition can be found by studying the fundamental group of the surface complement. Notice that applying a ribbon twist preserves the fundamental group of the surface complement. Therefore, if two surfaces $\Sigma_1, \Sigma_2$ are twist equivalent then the fundamental groups of their surface complements are isomorphic, $\pi_1(\mathbb{R}^3 \setminus \Sigma_1) \simeq \pi_1(\mathbb{R}^3 \setminus \Sigma_2)$. 

\begin{proof}[Proof of Proposition \ref{prop:twist_trivial_free}]
Let $\Sigma'$ be the standard n-fold torus with m holes to which $\Sigma$ is twist equivalent. From the fact that neither a ribbon twist nor an isotopy changes the fundamental group $\pi_1(\mathbb{R}^3 \setminus \Sigma)$ it follows that $\pi_1(\mathbb{R}^3 \setminus \Sigma) \simeq \pi_1(\mathbb{R}^3 \setminus \Sigma')$. The Seifert-van Kampen theorem (§70 in \cite{Munkres2000} or 1.2 in \cite{Hatcher2001}) implies that $\pi_1(\mathbb{R}^3 \setminus \Sigma') \simeq F_{2n+m-1}$, the free group of $2n+m-1$ generators.
\end{proof}

From this necessary condition it is straightforward to construct a surface that cannot be twist trivial. Many satellite knots and links (\cite{Cromwell2004} or \cite{Adams1994}) serve as counterexamples. For instance, take $L=L_1 \cup L_2$ to be the cable link of the trivial two component link with the trefoil knot as its companion (Figure \ref{fig:free_group_example}). A possible, Seifert surface $\Sigma$ is shown in Figure \ref{fig:free_group_example}. Clearly, the fundamental group of its complement $\pi_1(\mathbb{R}^3 \setminus \Sigma)$ is isomorphic to the knot group of the trefoil knot, which is known to be non-free (Chapter 3 Section B in \cite{Rolfsen1976}). Thus, the surface $\Sigma$ cannot be a canonical Seifert surface (\cite{Kawauchi1996}) nor can it be twist trivial. This is also intuitively obvious for this example because a ribbon twist can only add full twists to the band. However, this will obviously never unknot the surface and so it cannot be twist trivial.

\begin{figure}[htp]
  \centering
  \begin{overpic}{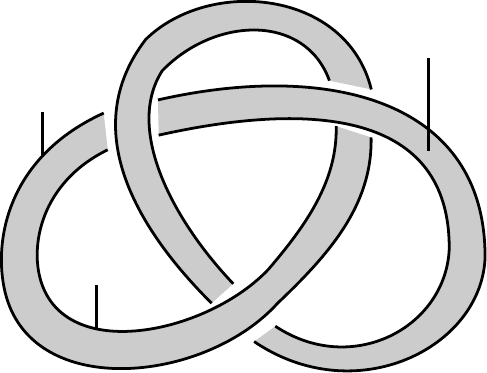}
    \put(4,55){$L_1$}
    \put(15,20){$L_2$}
    \put(84,66){$\Sigma$}
  \end{overpic}
  \caption{}
  \label{fig:free_group_example}
\end{figure}

Baader and Dehornoy showed that a canonical Seifert surface of a positive braid knot is twist trivial (\cite{Baader2013}). This implies that a canonical Seifert surface of any braid knot is twist trivial. The reason is that a negative crossing can be changed into a positive crossing by a single ribbon twist. Therefore, we first make the braid knot positive by ribbon twists. Then, by Baader and Dehornoy's result we know that the braid knot is twist trivial. We generalize this result in Theorem \ref{thm:canonical_surface_twist_trivial} to canonical Seifert surfaces of any link. The remainder of this section is concerned with the proof of Theorem \ref{thm:canonical_surface_twist_trivial} and for that we need the following lemma. 

\begin{lemma}
\label{lem:ribbon_twist_isotopy_commute}
Let $\Sigma$ be a Seifert surface, $H:[0,1]\times\mathbb{R}^3 \to \mathbb{R}^3$ an ambient isotopy and $R$ a ribbon twist on $\Sigma$. Then, there exists a ribbon twist $R'$ on $H(1,\Sigma)$ such that
\begin{equation*}
H(1,R(\Sigma)) =R'(H(1,\Sigma)).
\end{equation*}
\end{lemma}
\begin{proof}
The ribbon twist $R$ is a local cut and glue operation along an interval $I \subset \Sigma$. The interval $I$ is deformed along with the surface $\Sigma$ by the isotopy $H$. So, applying a ribbon twist $R'$ to $H(1,\Sigma)$ along the interval $H(1,I)$ is the same as applying the ribbon twist $R$ along $I$ before the isotopy and then applying the isotopy $H$ (Figure \ref{fig:isotopy_rt_commute}).
\end{proof}
\begin{figure}[htp]
  \centering
  \begin{overpic}{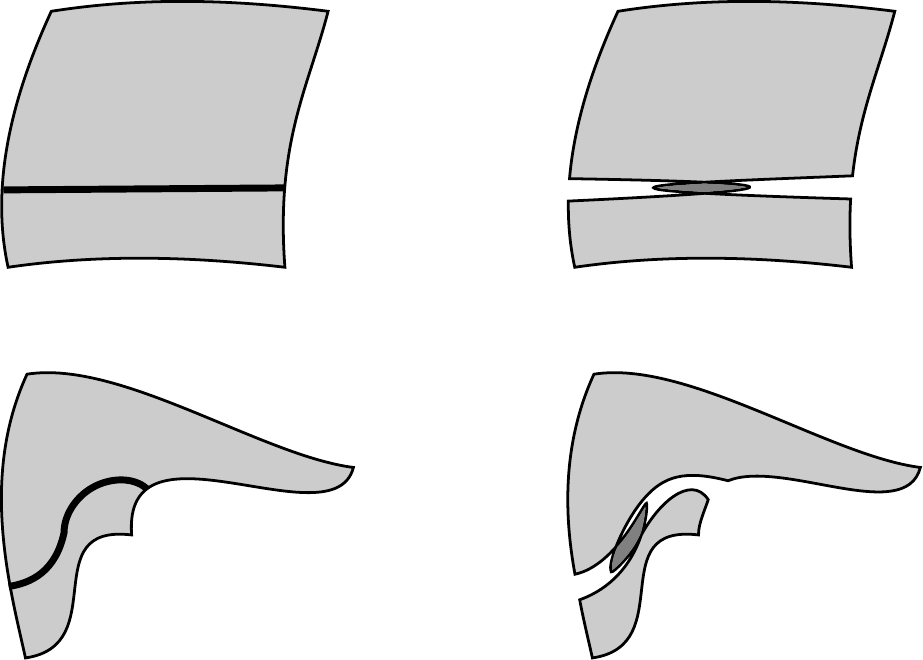}
    \put(7,63){$\Sigma$}
    \put(3,52){$I$}
    \put(-5,34){$H(1,\Sigma) \Downarrow $}
    \put(41,61){$R(\Sigma)$}
    \put(43,54){$\Rightarrow$}
    \put(65,34){$\Downarrow H(1,R(\Sigma))$}
    \put(43,13){$\Rightarrow$}
    \put(21,5){$R'(H(1,\Sigma))$}
    \put(42,34){$\circlearrowleft$}
  \end{overpic}
  \caption{}
  \label{fig:isotopy_rt_commute}
\end{figure}

The previous lemma helps us to show that two surfaces $\Sigma, \Sigma'$ are twist equivalent. Now, we can make use of isotopies and ribbon twists in any order. After we found a sequence of isotopies and ribbon twists which maps $\Sigma$ to $\Sigma'$ we can apply the lemma to reorder the sequence such that all ribbon twists are applied first. The resulting surface of the ribbon twists is isotopic to $\Sigma'$ by the remaining isotopies of the sequence. Hence, $\Sigma$ and $\Sigma'$ are twist equivalent. 

The fact that we can deform surfaces by isotopies between ribbon twists to show twist equivalence leads to a procedure we call \emph{non-entangling band sliding}. Let $\Sigma = \left\{D_1,\dots,D_n;B_1,\dots,B_m;-\right\}$ be a type II free canonical Seifert surface. As observed earlier, $\Sigma$ is almost planar. Hence, it is possible to slide one end of a band $B_i \in \Sigma$ along $\partial \Sigma$ such that it does not entangle with any other bands. While sliding the end of $B_i$ along the boundary, we can either keep the band above or below the plane in which we can embed $\Sigma$. We might need to apply ribbon twists to the bands over which the end is slid. In the example in Figure \ref{fig:entangle}, if the $D_i$-end of $B_i$ is slid along the dashed arrow to $D_l$ then $B_i$ would entangle around $D_k$ because $B_j$ has an unsuitable sign. Applying a ribbon twist to change the sign of $B_j$ produces a path along the dashed arrow which keeps $B_i$ on top of all discs. Hence, non-entangling band sliding is possible. This procedure consists of successive ribbon twists and isotopies. By Lemma \ref{lem:ribbon_twist_isotopy_commute}, the surface after the non-entangling band sliding is twist equivalent to $\Sigma$ .

\begin{figure}[t]
  \centering
  \begin{overpic}{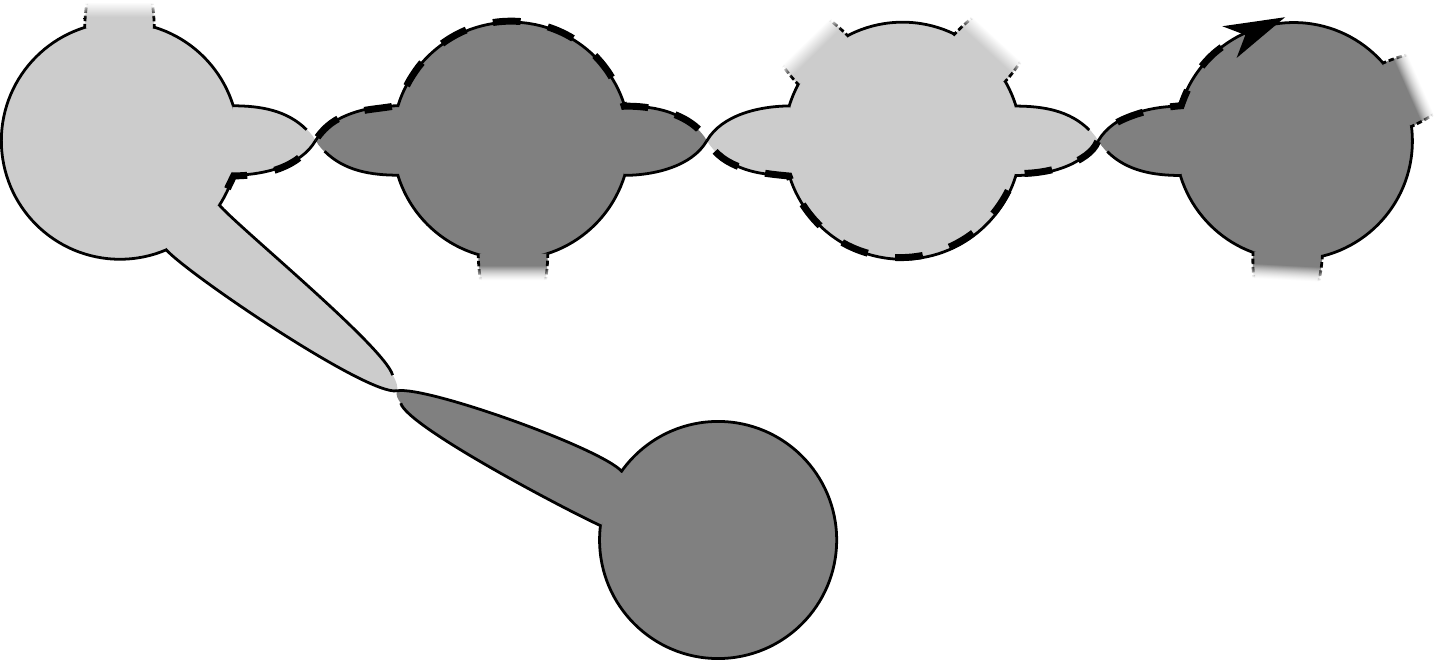}
    \put(47.5,40){$B_j$}
    \put(22,15){$B_i$}
    \put(6.5,35){$D_i$}
    \put(60.5,35){$D_k$}
    \put(88,35){$D_l$}
  \end{overpic}
  \caption{} 
  \label{fig:entangle}
\end{figure}

\begin{proof}[Proof of Theorem \ref{thm:canonical_surface_twist_trivial}]
The proof strategy is an induction over the number of half twisted bands. In the process of removing bands, the Seifert discs may become punctured or receive handles. These holes and handles can be considered infinitesimally small during the iteration steps. Therefore, they do not interact with any successive band removal.

The canonical Seifert surface may consist of $N$ components which are connected by $N-1$ unknotted tubes. By interpreting each component as a vertex and each tube as an edge, the components and unknotted tubes form a graph. In particular, they do not form any circuit. This implies that the graph is a tree. Due to band removals, the canonical Seifert surface may split into new components. We assure that a new component is connected to the previous components by one unknotted tube. This implies that there are $N$ unknotted tubes connecting $N+1$ components, which means that the induced graph remains a tree.

The tree structure of the components and the unknotted tubes implies that we may consider all unknotted tubes and all but one components to be infinitesimally small during the iteration steps. Therefore, while removing bands from one component, all other components and unknotted tubes are infinitesimally small and do not interact with the band removal on this component. 

Before removing any bands, remove all type II Seifert discs according to Theorem \ref{thm:altonen}. We obtain a canonical Seifert surface $\Sigma'=\left\{D_1,\dots,D_n;B_1,\dots,B_m;T_1,\dots,T_o\right\}$ where $D_1, \dots, D_n$ are of type I. The surfaces $\Sigma$ and $\Sigma'$ are isotopic and therefore twist equivalent. The induction is now carried out on the type II free canonical Seifert surface $\Sigma'$.\\
\textbf{Base case:} Let $\Sigma'=\left\{D_1,\dots,D_n;-;T_1,\dots,T_{n-1}\right\}$ where the discs $\left\{D_1,\dots,D_n\right\}$ are the components connected by $\left\{T_1,\dots,T_{n-1}\right\}$. Every hole and handle can be moved to disc $D_1$ over the connecting tubes by an isotopy. As already mentioned, the induced graph of the components and unknotted tubes is a tree without multi-edges. In particular, $D_1$ can be interpreted as the root of the tree. Successively, the tree can be reduced from the leaves. A leaf is a disc with only one connecting tube. By retracting the tube, the leaf disc becomes a hole in its neighboring disc (Figure \ref{fig:tube_retraction}). This new hole can again be moved to $D_1$. Repeating these steps leads to a single disc with holes and handles. Thus, $\Sigma'$ is twist trivial.
\begin{figure}[htbp]
  \centering
  \begin{subfigure}[t]{0.3\textwidth}
    \centering
    \includegraphics{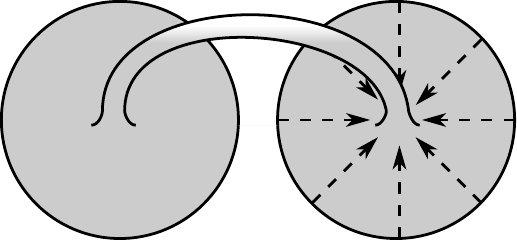}
  \end{subfigure}
  \begin{subfigure}[t]{0.3\textwidth}
    \centering
    \includegraphics{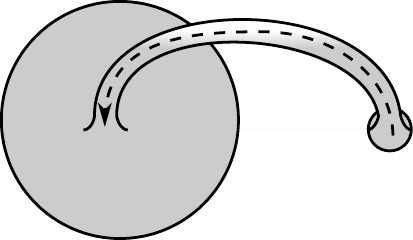}
  \end{subfigure}
  \begin{subfigure}[t]{0.3\textwidth}
    \centering
    \includegraphics{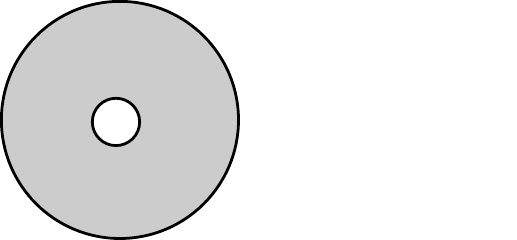}
  \end{subfigure}
  \caption{}
  \label{fig:tube_retraction}
\end{figure}
\\
\textbf{Inductive step:} Let $\Sigma'=\left\{D_1,\dots,D_n;B_1,\dots,B_m;T_1,\dots,T_o\right\}$. There must be at least one component which consists of at least two Seifert discs and at least one half twisted band. As mentioned, we consider all other components and unknotted tubes to be infinitesimally small. This allows the application of Lemma \ref{lem:graph_limitation} to this component. Therefore, at least one of the following cases occurs.
\\
\emph{Case 1:} $\exists k,l \in \left\{1,\dots,m\right\}$ such that $B_k, B_l$ are adjacent. \\
If two half twisted bands are adjacent then they connect the same discs $D_i, D_j \in \Sigma'$ and there does not exist any other band between $B_k, B_l$. If the signs of the adjacent bands $B_k, B_l$ are the same then the sign of $B_k$ or $B_l$ can be changed by a ribbon twist. This results in $B_k$ and $B_l$ having different signs. Subsequently, $B_k$ and $B_l$ can be removed by an isotopy, which is in this case a Reidemeister II move (Figure \ref{fig:reidemeister_move_a}). This results in an unknotted tube $T$ which connects $D_i, D_j$ (Figure \ref{fig:reidemeister_move_b}).
\begin{figure}[b]
  \centering
  \begin{subfigure}[t]{0.4\textwidth}
    \centering
    \begin{overpic}{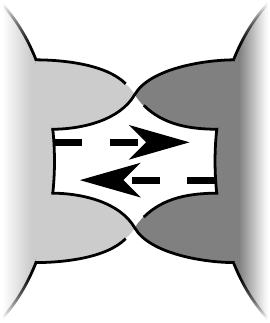}
      \put(-14,45){$D_i$}
      \put(77,45){$D_j$}
      \put(33,5){$B_k$}
      \put(33,83){$B_l$}
    \end{overpic}
    \caption{}
    \label{fig:reidemeister_move_a}
  \end{subfigure}
  \begin{subfigure}[t]{0.4\textwidth}
    \centering
    \begin{overpic}{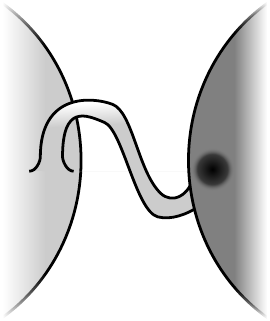}
      \put(-14,45){$D_i$}
      \put(77,45){$D_j$}
      \put(40,65){$T$}
    \end{overpic}
    \caption{}
    \label{fig:reidemeister_move_b}
  \end{subfigure}
  \caption{}
\end{figure}

Next, there are two cases. First case, the new tube $T$ is the only connection between otherwise disconnected components. This is the case where one new component arises. Notice that this new component is connected only via $T$. Hence, $\left\{D_1,\dots,D_n;B_1,\dots,B_m;T_1,\dots,T_o, T\right\} \setminus \left\{B_k, B_l\right\}$ is the resulting surface with now $o+2$ components.

\begin{figure}[b]
  \centering
  \begin{subfigure}[t]{0.45\textwidth}
    \centering
    \begin{overpic}{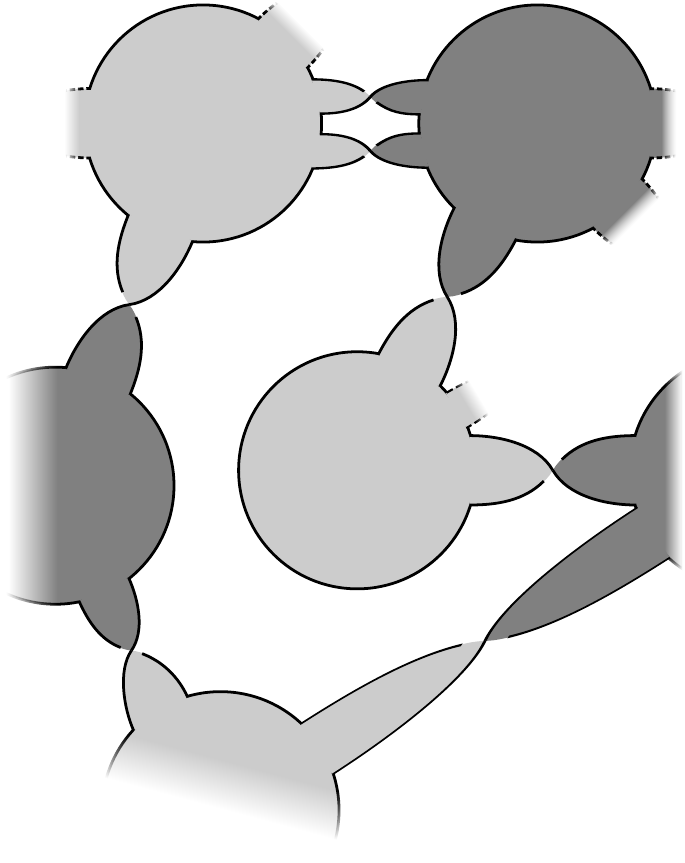}
      \put(40,93){$B_l$}
      \put(40,74){$B_k$}
      \put(20,90){$D_i$}
      \put(60,90){$D_j$}
    \end{overpic}
    \caption{}
    \label{fig:thm_tube_to_handle_circuit_a}
  \end{subfigure}
  \begin{subfigure}[t]{0.45\textwidth}
    \centering
    \begin{overpic}{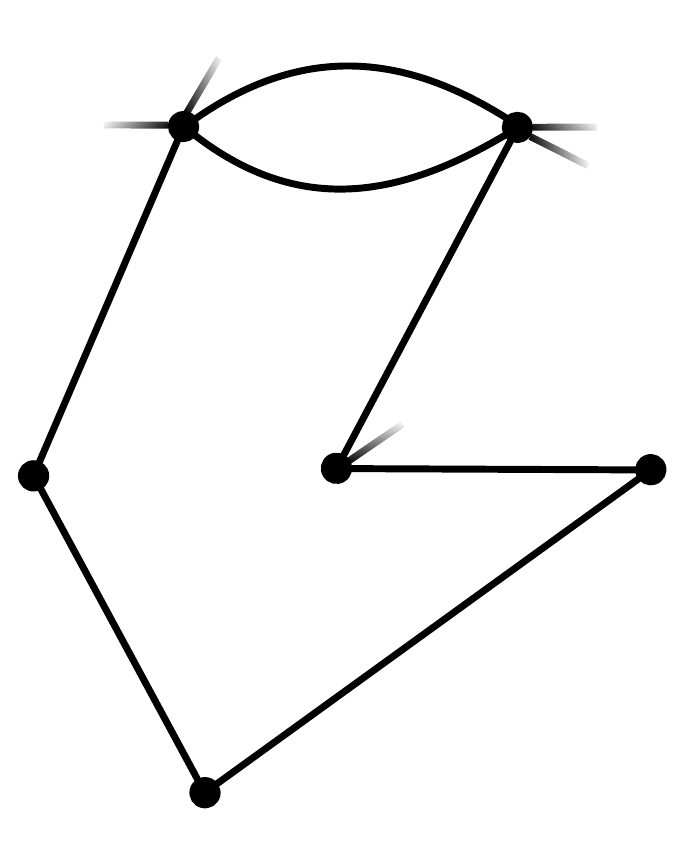}
      \put(53,90){$e_l$}
      \put(48,83){$e_k$}
      \put(35,83){$f_{lk}$}
      \put(38,100){$f_l$}
      \put(23,53){$f_k$}
    \end{overpic}
    \caption{}
    \label{fig:thm_tube_to_handle_circuit_b}
  \end{subfigure}
  \begin{subfigure}[t]{0.45\textwidth}
    \centering
    \begin{overpic}{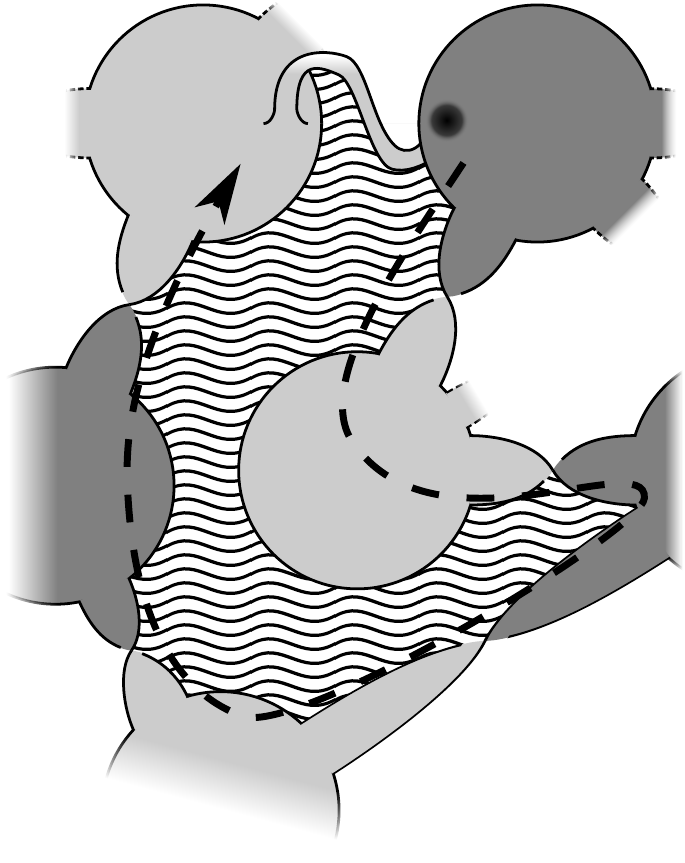}
      \put(43,93){$T$}
      \put(20,90){$D_i$}
      \put(60,90){$D_j$}
    \end{overpic}
    \caption{}
    \label{fig:thm_tube_to_handle_circuit_c}
  \end{subfigure}
  \begin{subfigure}[t]{0.45\textwidth}
    \centering
    \includegraphics{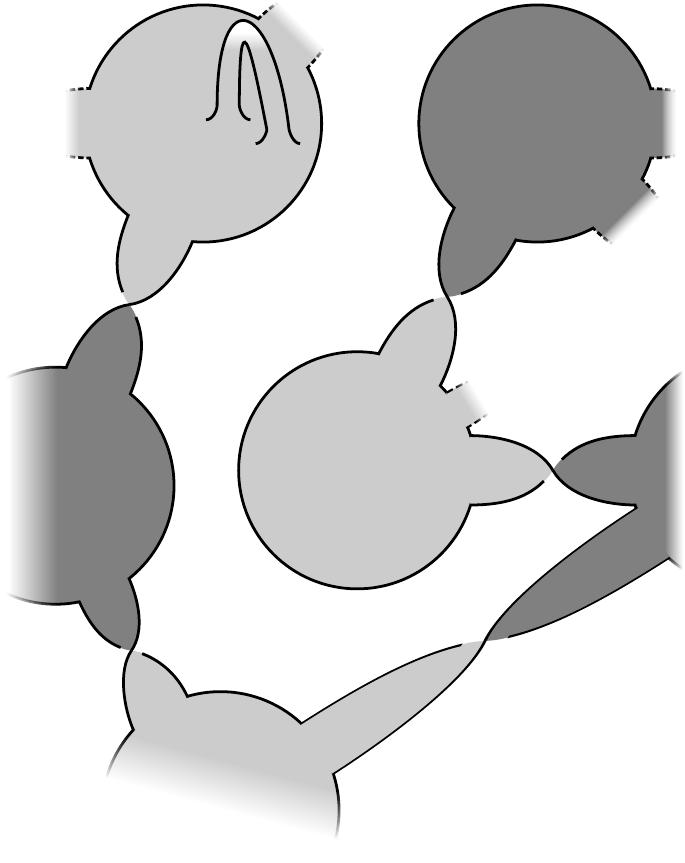}
    \caption{}
    \label{fig:thm_tube_to_handle_circuit_d}
  \end{subfigure}
  \caption{}
  \label{fig:thm_tube_to_handle_circuit}
\end{figure}
The second case is when the current component is still connected even without the tube $T$. This means that there exists a path along the remaining half twisted bands in the component such that one end of the tube $T$ can be moved to its other end. Now, our goal is to show that we can even move one end of $T$ to its other end without entanglement. For this, consider the Seifert surface and its Seifert graph prior to the Reidemeister II move (Figure \ref{fig:thm_tube_to_handle_circuit_a}). The bands $B_l, B_k$ induce two edges $e_l, e_k$ in the Seifert graph. The edges $e_l, e_k$ bound a face $f_{kl}$ because $B_l, B_k$ are adjacent. Additionally, each $e_l$ and $e_k$ belong to the boundary of further two faces $f_l, f_k$. Note that $f_l \neq f_k$ because the current component is connected even without $T$ and so even without $B_l, B_k$ (Figure \ref{fig:thm_tube_to_handle_circuit_b}). Now, consider the boundary $\partial f_l$ or $\partial f_k$. Without loss of generality we continue with $\partial f_k$. There exist half twisted bands which induce the path $\partial f_k \setminus e_k$. These half twisted bands form a path from $D_i$ to $D_j$ in the current component and are the reason why the component is still connected even without $T$. In other words, these bands form a path from one end of the tube to its other end. Since we found this path through the boundary of the face $f_k$, the tube $T$ and these bands bound an area which does not contain any other disc, tube or band (shaded area in Figure \ref{fig:thm_tube_to_handle_circuit_c}). We can ensure by ribbon twists that the half twisted bands form a path with alternating signs such that one end of the tube $T$ can be moved to its other end without entanglement. The result is a handle on one of the discs $D_i, D_j$ (Figure \ref{fig:thm_tube_to_handle_circuit_d}).
\\
\emph{Case 2:} $\exists i \in \left\{1,\dots,n\right\}$ such that $deg(D_i) =1$ \\
The disc $D_i$ along with the attached half twisted band can be merged into the neighboring disc by an isotopy.
\\
\emph{Case 3:} $\exists i \in \left\{1,\dots,n\right\}$ such that $deg(D_i) =2$ \\
Let $B(D_i) = \left\{B_k, B_l\right\}$ be the two attached bands. In this case there are two subcases. The first case is when both bands are connected to the same disc $D_k$. It might be that $D_k$ has a group $G$ of bands and discs between $B_k, B_l$ making $B_k, B_l$ non-adjacent. However, $G$ can be slid away over $B_k$ or $B_l$ such that $B_k, B_l$ become adjacent (Figure \ref{fig:group_slide}). Then the adjacent bands can be removed as in case 1.

\begin{figure}[htp]
  \centering
  \begin{subfigure}[t]{0.45\textwidth}
    \centering
    \begin{overpic}{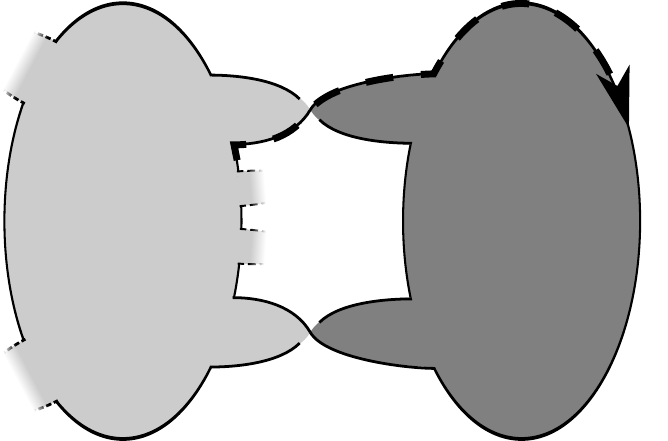}
      \put(43,56){$B_k$}
      \put(43,4){$B_l$}
      \put(14,30){$D_k$}
      \put(73,30){$D_i$}
      \put(40,30){$G$}
    \end{overpic}
  \end{subfigure}
  \begin{subfigure}[t]{0.45\textwidth}
    \centering
    \begin{overpic}{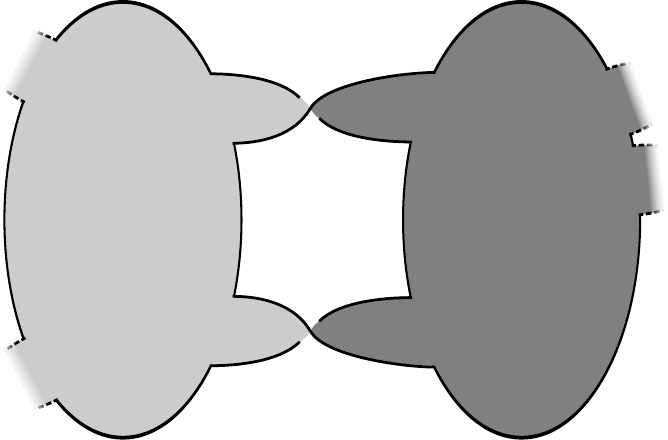}
      \put(43,56){$B_k$}
      \put(43,4){$B_l$}
      \put(14,30){$D_k$}
      \put(73,30){$D_i$}
      \put(100,45){$G$}
    \end{overpic}
  \end{subfigure}
  \caption{}
  \label{fig:group_slide}
\end{figure}

The second case is when $B_k, B_l$ connect to two different discs $D_k, D_l$. By a ribbon twist we can assure that $B_k, B_l$ have different signs. After a rotation of $D_i$, it is possible to merge $D_i, D_k, D_l, B_k, B_l$ by an isotopy into one disc (Figure \ref{fig:two_band_disc_removal}).
\begin{figure}[b]
  \centering
  \begin{subfigure}[t]{0.45\textwidth}
    \centering
    \begin{overpic}{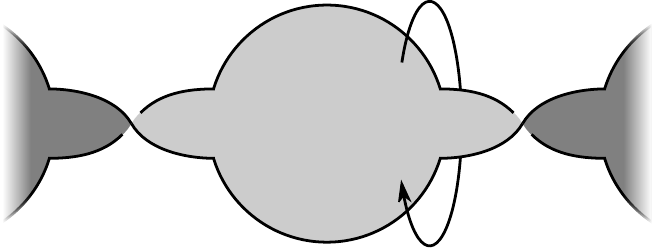}
      \put(16,26){$B_k$}
      \put(76,26){$B_l$}
      \put(45,16){$D_i$}
      \put(0,16){$D_k$}
      \put(92,16){$D_l$}
    \end{overpic}
  \end{subfigure}
  \begin{subfigure}[t]{0.45\textwidth}
    \centering
    \begin{overpic}{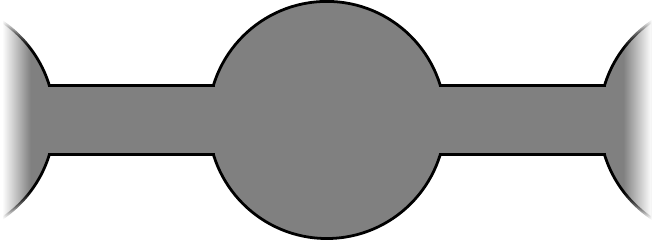}
      \put(16,26){$B_k$}
      \put(76,26){$B_l$}
      \put(45,16){$D_i$}
      \put(0,16){$D_k$}
      \put(92,16){$D_l$}
    \end{overpic}
  \end{subfigure}
  \caption{}
  \label{fig:two_band_disc_removal}
\end{figure}
\\
\emph{Case 4:} $\exists i \in \left\{1,\dots,n\right\}$ such that $deg(D_i) =3$ \\
Let $B(D_i) = \left\{B_k, B_l, B_m\right\}$ be the three attached bands. In this case there are three subcases. The first case is where $B_k, B_l, B_m$ are connected to the same disc. So the situation is analog to the first subcase of case 3 except that there is one additional band. However, this does not prevent us from removing a pair of bands in the same way as in case 3. Therefore, if a pair of bands is adjacent they can be removed. If both pairs are non adjacent either one of them can be made adjacent as in case 3 by sliding away the group of bands and discs making the pair non-adjacent (Figure \ref{fig:group_slide}).

The second case is when the $B_k, B_l, B_m$ are connected to two different discs $D_k, D_l$. Without loss of generality let $B_k$ be connected to $D_k$ and $B_l, B_m$ to $D_l$. If $B_l, B_m$ are adjacent then we can follow case 1. If $B_l, B_m$ are non-adjacent due to a group of bands and discs which are only attached to either $D_i$ or $D_l$ (but not both) we can slide the group away as in case 3. If $B_l, B_m$ are non-adjacent due to a group $G$ of bands and discs which are attached to both $D_i$ and $D_l$ we can exploit that $G$ is attached to $D_i$ only via $B_k$ (Figure \ref{fig:three_band_almost_adjacent_a}). Consider $G$ to be a rigid plane except for $B_k$. First, slide $G$ over $B_l$. Then apply a ribbon twist to $B_l$ in order to avoid entanglement in the next slide (Figure \ref{fig:three_band_almost_adjacent_b}). Then slide the $D_i$-end of $B_k$ over $B_l$. After the sliding, $B_k$ might be twisted one and a half times but by a ribbon twist we obtain a half twisted band. Eventually, $B_l, B_m$ become a pair of adjacent bands, which can be removed as in case 1 (Figure \ref{fig:three_band_almost_adjacent_c}). 

\begin{figure}[htp]
  \centering
  \begin{subfigure}[t]{0.45\textwidth}
    \centering
    \begin{overpic}{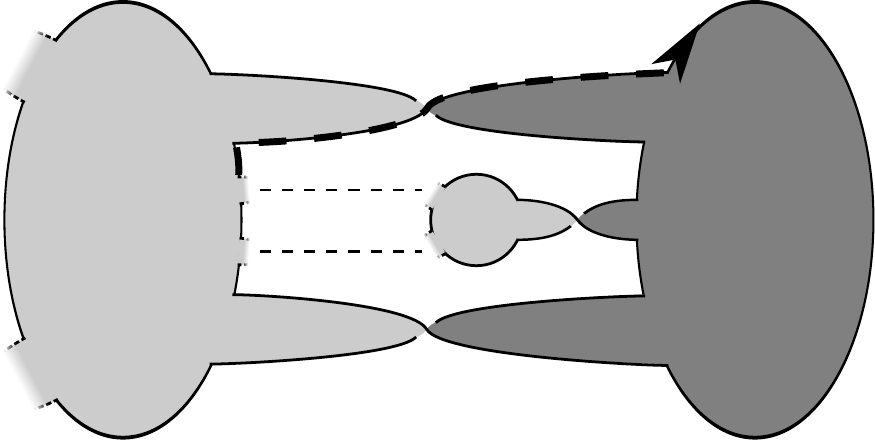}
      \put(45,42){$B_l$}
      \put(45,4){$B_m$}
      \put(61,28){$B_k$}
      \put(11,23){$D_l$}
      \put(83,23){$D_i$}
      \put(50,23){$D_k$}
      \put(35,22.5){$G$}
    \end{overpic}
    \caption{}
    \label{fig:three_band_almost_adjacent_a}
  \end{subfigure}
  \begin{subfigure}[t]{0.45\textwidth}
    \centering
    \begin{overpic}{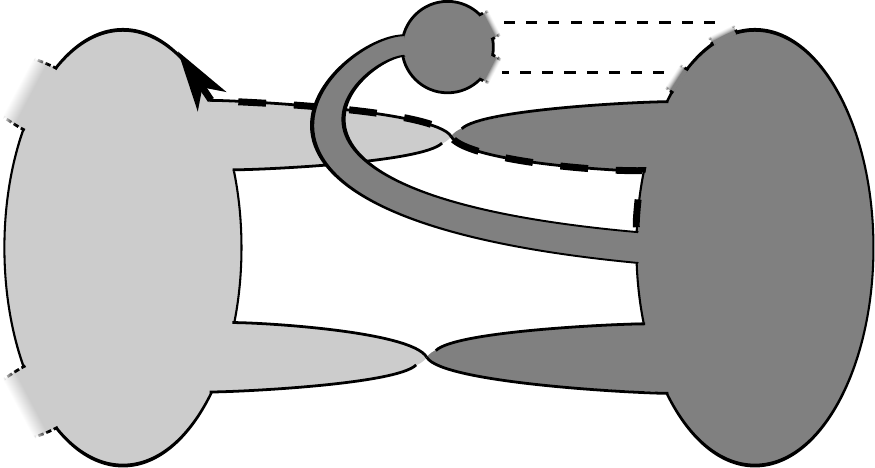}
      \put(25,44){$B_l$}
      \put(45,4){$B_m$}
      \put(62,28){$B_k$}
      \put(11,23){$D_l$}
      \put(83,23){$D_i$}
      \put(46.5,45.5){$D_k$}
      \put(65,45.5){$G$}
    \end{overpic}
    \caption{}
    \label{fig:three_band_almost_adjacent_b}
  \end{subfigure}
  \begin{subfigure}[t]{0.45\textwidth}
    \centering
    \begin{overpic}{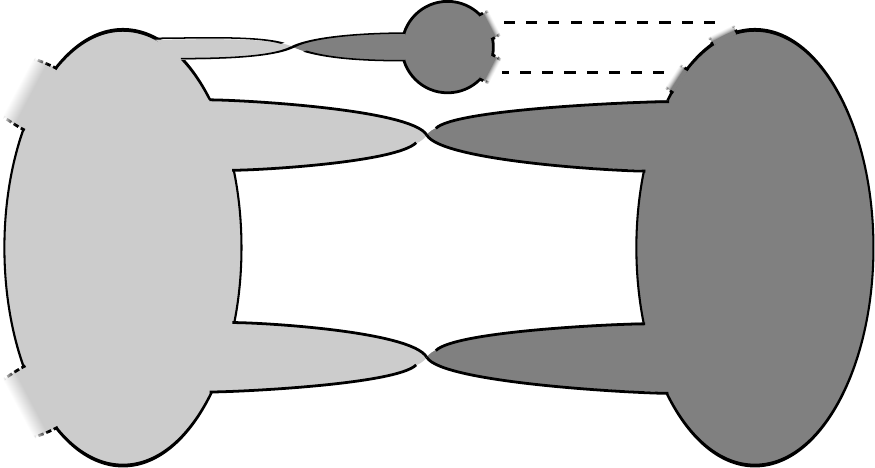}
      \put(30,28){$B_l$}
      \put(45,4){$B_m$}
      \put(29,51){$B_k$}
      \put(11,23){$D_l$}
      \put(83,23){$D_i$}
      \put(46.5,45.5){$D_k$}
      \put(65,45.5){$G$}
    \end{overpic}
    \caption{}
    \label{fig:three_band_almost_adjacent_c}
  \end{subfigure}
  \caption{}
  \label{fig:three_band_almost_adjacent}
\end{figure}

\begin{figure}[b]
  \centering
  \begin{subfigure}[t]{0.45\textwidth}
    \centering
    \begin{overpic}{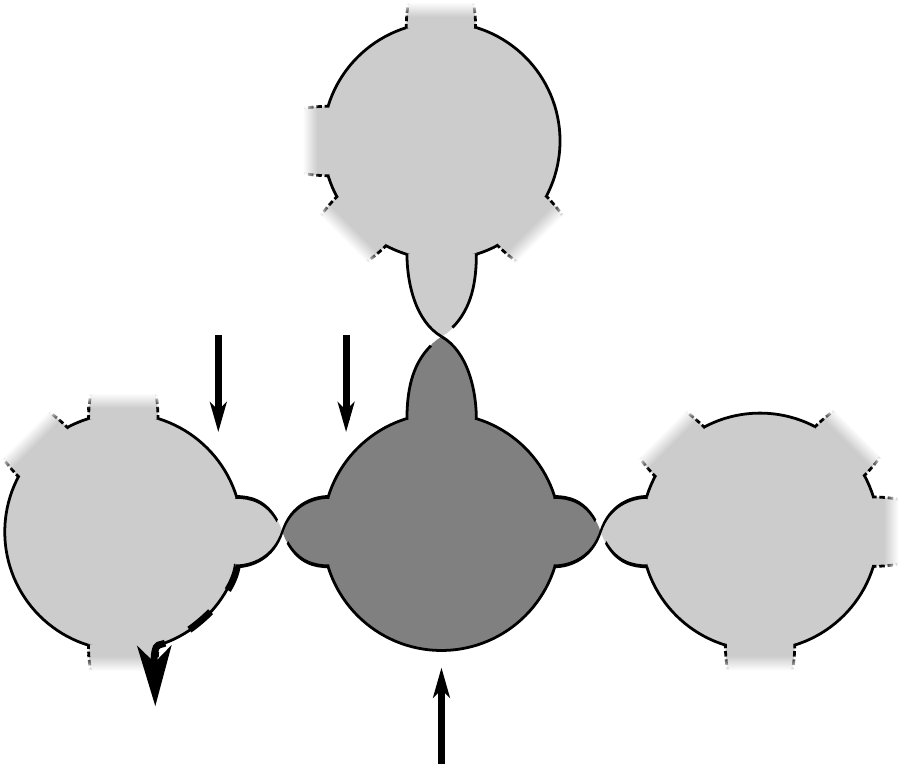}
      \put(16,43){$\alpha)$}
      \put(30,43){$\beta)$}
      \put(41,3){$\gamma)$}
      \put(45,24){$D_i$}
      \put(10,24){$D_k$}
      \put(45,67){$D_l$}
      \put(80,24){$D_m$}
      \put(27,31){$B_k$}
      \put(53,43){$B_l$}
      \put(61,31){$B_m$}
    \end{overpic}
    \caption{}
    \label{fig:three_band_slide_end_positions}
  \end{subfigure}  
  \begin{subfigure}[t]{0.45\textwidth}
    \centering
    \begin{overpic}{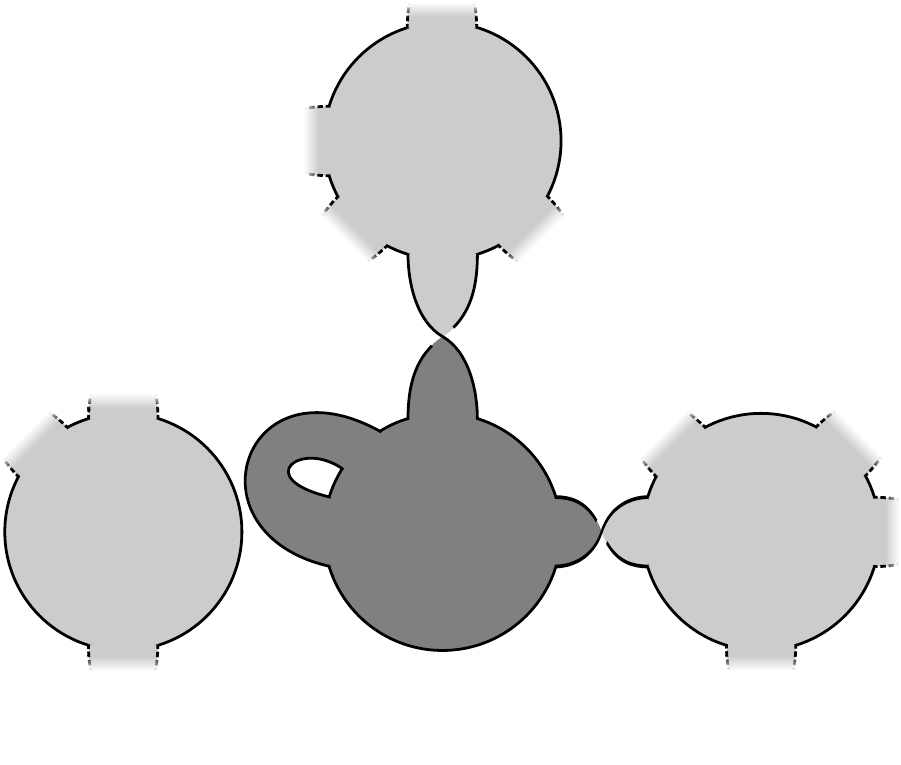}
      \put(45,24){$D_i$}
      \put(10,24){$D_k$}
      \put(45,67){$D_l$}
      \put(80,24){$D_m$}
      \put(27,41){$B_k$}
      \put(53,43){$B_l$}
      \put(61,31){$B_m$}
    \end{overpic}
    \caption{}
    \label{fig:three_band_slide_end_position_b}
  \end{subfigure}  
  \begin{subfigure}[t]{0.45\textwidth}
    \centering
    \begin{overpic}{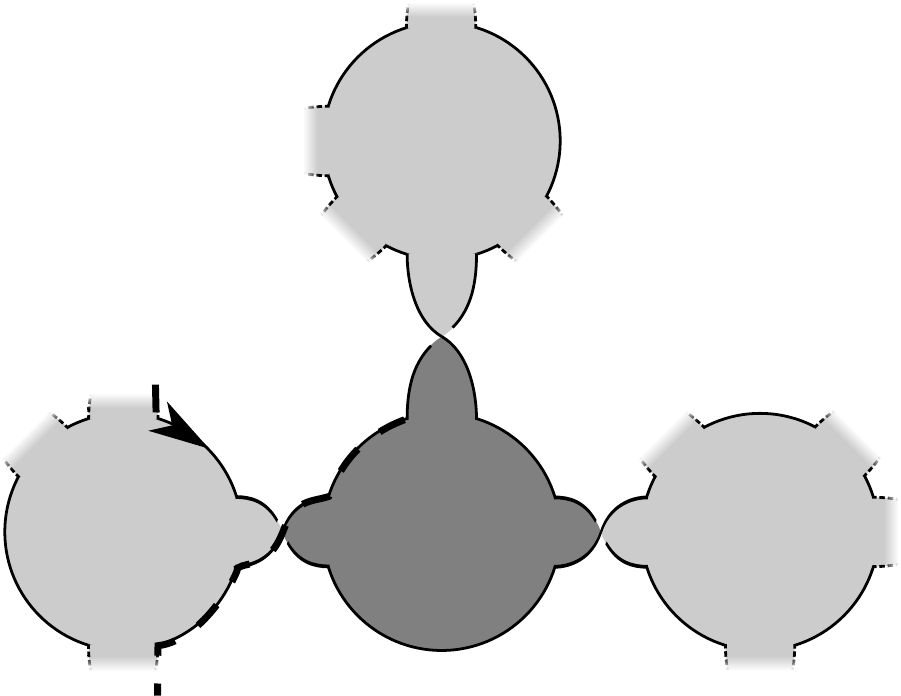}
      \put(45,16){$D_i$}
      \put(10,16){$D_k$}
      \put(45,59){$D_l$}
      \put(80,16){$D_m$}
      \put(27,23.5){$B_k$}
      \put(53,35){$B_l$}
      \put(61,23.5){$B_m$}
    \end{overpic}
    \caption{}
    \label{fig:three_band_slide_end_position_a_a}
  \end{subfigure}
  \begin{subfigure}[t]{0.45\textwidth}
    \centering
    \begin{overpic}{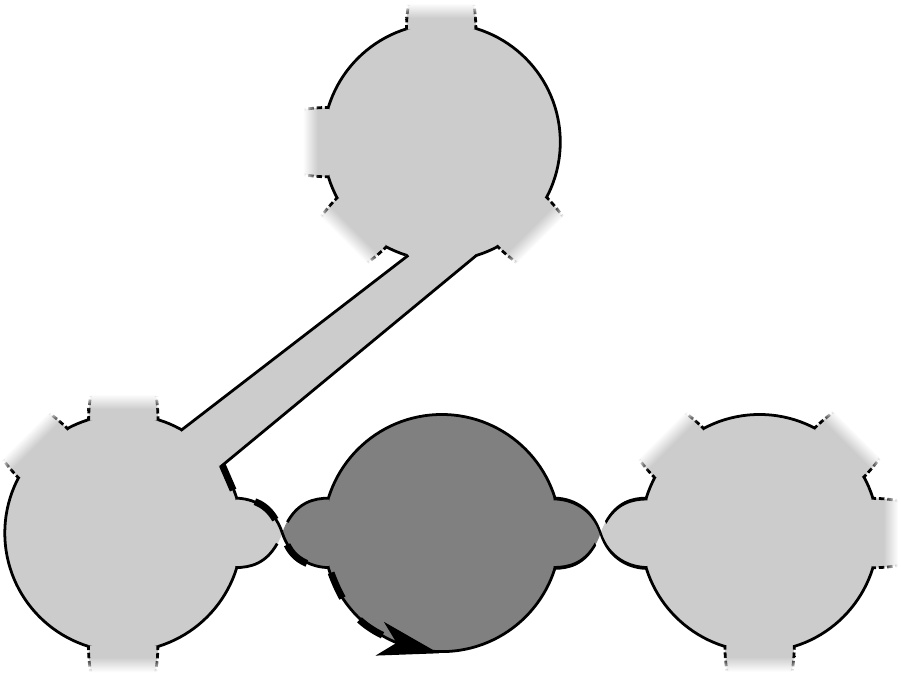}
      \put(45,16){$D_i$}
      \put(10,16){$D_k$}
      \put(45,59){$D_l$}
      \put(80,16){$D_m$}
      \put(27,23.5){$B_k$}
      \put(25,41){$B_l$}
      \put(61,23.5){$B_m$}
    \end{overpic}
    \caption{}
    \label{fig:three_band_slide_end_position_a_b}
  \end{subfigure}
  \begin{subfigure}[t]{0.45\textwidth}
    \centering
    \begin{overpic}{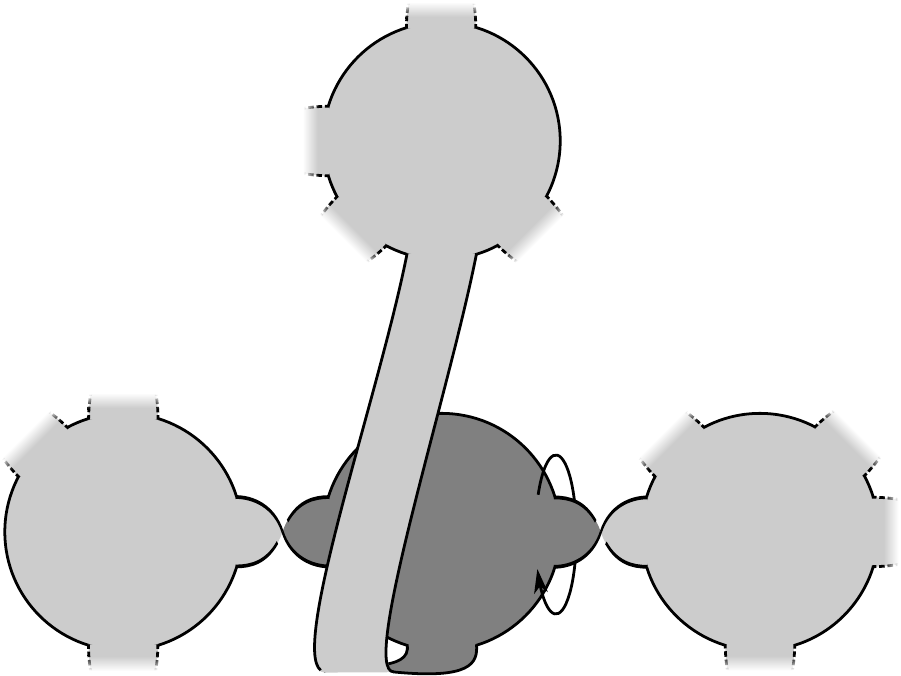}
      \put(46,13){$D_i$}
      \put(10,13){$D_k$}
      \put(45,57){$D_l$}
      \put(80,13){$D_m$}
      \put(27,21){$B_k$}
      \put(34,35){$B_l$}
      \put(61,21){$B_m$}
    \end{overpic}
    \caption{}
    \label{fig:three_band_slide_end_position_a_c}
  \end{subfigure}
  \begin{subfigure}[t]{0.45\textwidth}
    \centering
    \begin{overpic}{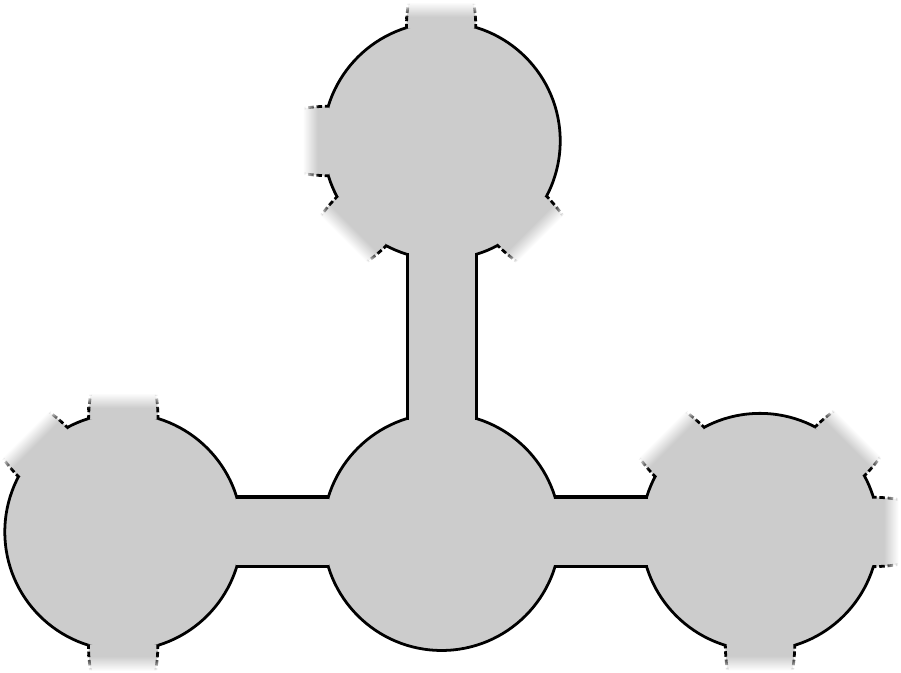}
      \put(46,13){$D_i$}
      \put(10,13){$D_k$}
      \put(45,57){$D_l$}
      \put(80,13){$D_m$}
      \put(27,21){$B_k$}
      \put(36,35){$B_l$}
      \put(61,21){$B_m$}
    \end{overpic}
    \caption{}
    \label{fig:three_band_slide_end_position_a_d}
  \end{subfigure}
  \caption{}
  \label{fig:three_band_slide_end_position_a}
\end{figure}

The third case is when $B_k$, $B_l$, $B_m$ are connected to three different discs $D_k$, $D_l$, $D_m$. In order to determine how to reduce the number of half twisted bands, slide the $D_k$-end of $B_k$ along the boundary in the same direction until the end of the band reaches one of the three positions $\alpha),\beta)$ or $\gamma)$ as shown in Figure \ref{fig:three_band_slide_end_positions}. The reason that exactly one of $\alpha),\beta),\gamma)$ is reached is that we can continue sliding the end of $B_k$ until the other end of $B_k$ blocks the path. In other words, position $\beta)$ or $\gamma)$ is reached. If neither $\beta)$ nor $\gamma)$ is reached, we can continue sliding forever. This can only happen if we reach the point where we started. So, in this case position $\alpha)$ is reached.

If $\alpha)$ is reached then there is a path from the $D_i$-end of $B_l$ to $\alpha)$ (Figure \ref{fig:three_band_slide_end_position_a_a}). So, by crossing $B_k$ a second time there is even a path from $\beta)$ to $\gamma)$. By applying ribbon twists and isotopies, the $D_i$-end of $B_l$ can be slid to position $\gamma)$ without entanglement (Figure \ref{fig:three_band_slide_end_position_a_b}). Now $B_l$ is again attached to $D_i$ but to the opposite side of its starting position (Figure \ref{fig:three_band_slide_end_position_a_c}). The band $B_l$ might have several full twists (at most as many as the number of crossings $B_l$ crossed) but these can be removed by ribbon twists. With at most two ribbon twists we can change the crossing signs of $B_k, B_m$ such that the entire disc $D_i$ can be rotated into a convenient position (Figure \ref{fig:three_band_slide_end_position_a_d}). From this position $B_k, B_l, B_m$ and $D_i, D_k, D_l, D_m$ can be merged into one disc.

If $\beta)$ or $\gamma)$ is reached then the $D_k$-end of $B_k$ can be slid to $\beta)$ or $\gamma)$ without entanglement. This results in one additional hole in $D_i$ (Figure \ref{fig:three_band_slide_end_position_b}). 
\end{proof}

\section{Untwisting Number}
The untwisting number of a surface is invariant under isotopy. This follows directly from Lemma \ref{lem:ribbon_twist_isotopy_commute}.  Just like the unknotting number, the untwisting number is very difficult to compute (\cite{Bleiler1984,Wendt1937}). Nevertheless, a lower bound can be given by the unknotting number. 
\begin{remark}
\label{cor:untwisting_lower}
Let $\Sigma$ be twist trivial. Then,
\begin{equation*}
u(\partial \Sigma) \leq ut(\Sigma).
\end{equation*}
\end{remark}

This lower bound is optimal in the sense that there exist twist trivial surfaces for which the equality holds. Baader and Dehornoy showed that a canonical Seifert surface $\Sigma$ of a positive braid knot can be untwisted by $g$ ribbon twists where $g$ is the genus of the surface (\cite{Baader2013}). Furthermore, Rudolph showed that the unknotting number of positive braid knots is not smaller than $g$ (\cite{Rudolph1998}). Thus, we know $ut(\Sigma) = g$ from Baader and Dehornoy, $u(\partial \Sigma) \geq g$ from Rudolph and $u(\partial \Sigma) \leq ut(\Sigma)$ from the remark. This implies that $u(\partial \Sigma) = ut(\Sigma) = g$.
\\

The proof of Theorem \ref{thm:canonical_surface_twist_trivial} is basically an algorithm which untwists any canonical Seifert surface. Therefore, we can derive an upper bound of the untwisting number for canonical Seifert surfaces.

\begin{proof}[Proof of Corollary \ref{cor:upper_bound}]
From the proof of Theorem \ref{thm:canonical_surface_twist_trivial} we see that we need to analyze four cases for a type II free canonical Seifert surface. 
\\
\emph{Case 1:} $\exists k,l \in \left\{1,\dots,m\right\}$ such that $B_k, B_l$ are adjacent. \\
At most, one ribbon twist is required to make the signs of $B_k,B_l$ unequal. Then, if the Reidemeister II move leads to a new component we are done. In the other case, we need at most $m-2$ ribbon twists to move one end of the new tube to its other end (Figure \ref{fig:thm_tube_to_handle_circuit}). 

This results in at most $m-1$ ribbon twists in order to remove two bands.
\\
\emph{Case 2:} $\exists i \in \left\{1,\dots,n\right\}$ such that $deg(D_i) =1$ \\
No ribbon twist is required to remove one band.
\\
\emph{Case 3:} $\exists i \in \left\{1,\dots,n\right\}$ such that $deg(D_i) =2$ \\
If both bands connect to the same two discs then this is equivalent to case 1 because we remove two adjacent bands.

If each band connects to a different disc, then we can remove two bands with at most one ribbon twist (Figure \ref{fig:two_band_disc_removal}). 

This results in $m-1$ ribbon twists in order to remove two bands in the worst case.
\\
\emph{Case 4:} $\exists i \in \left\{1,\dots,n\right\}$ such that $deg(D_i) =3$ \\
Let $B(D_i) = \left\{B_k, B_l, B_m\right\}$ be the three attached bands. If all bands connect to the same disc then this is equivalent to case 1 because we remove two adjacent bands. Hence, we require at most $m-1$ ribbon twists to remove two bands.

If the bands connect two discs then we may need to slide away a group of bands (Figure \ref{fig:three_band_almost_adjacent}). This may require two ribbon twists. Then, we are in the case of adjacent bands again. Therefore, we need at most $m+1$ ribbon twists in order to remove two bands.

If each band connects to a separate disc we can remove either one or three bands. 
If only one band can be removed (Figure \ref{fig:three_band_slide_end_position_b}) then we need at most $2(m-1)$ ribbon twists. This is because we need to change the sign of a band twice if the end of $B_k$ is slid twice over this band. 
If all three bands can be removed (Figure \ref{fig:three_band_slide_end_position_a}) we need at most $2(m-2)$ ribbon twists to slide $B_l$ to the opposite side of $D_i$. This is because we never cross $B_l$ itself nor $B_m$ and we may cross bands twice. After $B_l$ is slid to the opposite side of $D_i$ it may be twisted at most $m-2$ times. To untwist $B_l$ we need $m-2$ ribbon twists. Finally, we need at most two ribbon twists to enable the rotation of $D_i$. This adds up to at most $3(m-2)+2$ ribbon twists to remove three bands.

Therefore, the case where all bands connect to a different disc and only one band can be removed is the worst, with $2(m-1)$ ribbon twists for one band.

If there are $m\geq 3$ bands then the overall worst case is where $2(m-1)$ ribbon twists are needed to remove one band. If there are only $m = 2$ we need obviously at most one ribbon twist. Hence, $2(m-1)$ is also an upper bound. In conclusion, for type II free canonical Seifert surfaces 
\begin{equation*}
ut(\Sigma) \leq \sum\limits_{i=1}^{m-1} 2(m-i) = \sum\limits_{i=1}^{m-1} 2i.
\end{equation*}
Lastly, the isotopy to remove all type II Seifert discs increases the number of half twisted bands at most by factor three.
\end{proof}

\bigskip
\noindent
Universit\"at Bern, Sidlerstrasse 5, CH-3012 Bern, Switzerland

\bigskip
\noindent
\texttt{mpfeuti@ganymede.ch}

\end{document}